\documentclass{amsart}%
\usepackage{amsfonts}
\usepackage{amsmath}
\usepackage{amssymb}
\usepackage{graphicx}%
\providecommand{\U}[1]{\protect\rule{.1in}{.1in}}
\newtheorem{theorem}{Theorem}
\theoremstyle{plain}

\newtheorem{corollary}{Corollary}

\newtheorem{lemma}{Lemma}

\newtheorem{proposition}{Proposition}
\newtheorem{remark}{Remark}

\numberwithin{equation}{section}
\begin{document}
\title[Integrals of hypergeometric function]{A few identities and integrals involving Pochhammer symbols, Jacobi
polynomials, and the generalized hypergeometric function.}
\author{Pawe\l \ J. Szab\l owski}
\address{Department of Mathematics and Information Sciences, \\
Warsaw University of Technology\\
ul Koszykowa 75, 00-662 Warsaw, Poland}
\email{pawel.szablowski@gmail.com}
\date{March , 2025}

\begin{abstract}
We first present some identities involving the Pochhammer symbol (rising
factorial). We also recall and present some new properties of the Jacobi
polynomials. We use them to expand a general hypergeometric function in an
orthogonal series of Jacobi polynomials. Then we use these expansions to
discover closed forms for certain integrals of Jacobi polynomials that are
multiplied by a generalized hypergeometric function and a Beta density. We can
also obtain closed forms for some series involving rising factorials that
generalise binomial series by using well-known properties of the
hypergeometric function. In particular, we get a few new, nontrivial
identities involving the Pochhammer symbol. We can also derive some
simplifying identities for generalized hypergeometric functions.

\end{abstract}
\maketitle

\section{Introduction and Notation}

As stated in the abstract, the paper will expand a generalized hypergeometric
function in an orthogonal series of Jacobi polynomials and find some related
Beta integrals of the generalized hypergeometric function. As a byproduct, we
will obtain a few nontrivial identities involving the Pochhammer symbol (the
rising factorial).

Before proceeding with this expansion, we will briefly recall the basic
notions used in the sequel to complete the paper. Then we will obtain
specific, nontrivial, useful identities involving the Pochhammer symbol and
related to Jacobi polynomials. This is done in Section  \ref{Pochh}. We will
use the obtained identities in Section \ref{HYP} to get several expansions of
the hypergeometric function in the orthogonal series of Jacobi polynomials. We
can also derive some simplifying identities for hypergeometric functions, as
discussed at the end of Section \ref{HYP}.

We will extensively use the so-called rising factorial or Pochhammer symbol
(polynomial), which is defined by
\[
(x)^{(n)}=x(x+1)\ldots(x+n-1),
\]
for all complex $x$. Notice that we have for all $x\neq0$ we have%
\begin{equation}
(x)^{(n)}=\frac{\Gamma(x+n)}{\Gamma(x)}\text{,} \label{Gn}%
\end{equation}
where $\Gamma(x)$ denotes the Euler's gamma function. One can also consider
the so-called falling factorials, denoted by $\left(  a\right)  _{\left(
n\right)  }$ and defined by
\[
\left(  a\right)  _{\left(  n\right)  }\allowbreak=\allowbreak\prod
_{j=0}^{n-1}\left(  a-j\right)  ,
\]
with $\left(  a\right)  _{\left(  0\right)  }\allowbreak=\allowbreak1$.

Let us notice that we have
\begin{align*}
\left(  a\right)  ^{\left(  n\right)  }\allowbreak &  =\allowbreak\left(
-1\right)  ^{n}\left(  -a\right)  _{\left(  n\right)  },\\
\left(  a\right)  _{\left(  n\right)  }\allowbreak &  =\allowbreak\left(
-1\right)  ^{n}\left(  -a\right)  ^{\left(  n\right)  },
\end{align*}
and of course direcly from the definition we have
\begin{equation}
\left(  -x-n+1\right)  ^{\left(  n\right)  }=\left(  -1\right)  ^{n}\left(
x\right)  ^{\left(  n\right)  }. \label{-x}%
\end{equation}
Let us also recall the widely used extension of the binomial symbol, namely
for all complex $\alpha$ we define:%
\[
\binom{\alpha}{n}\overset{df}{=}\frac{\left(  \alpha\right)  _{\left(
n\right)  }}{n!}%
\]

The Pochhammer symbol is often denoted by the symbol $\left(  x\right)  _{n}$
is the literature. We will use the above notation as more intuitive and used
on Wikipedia. Moreover, within the so-called $q$-series theory, the symbol
$\left(  a\right)  _{n}$ denotes the $q$-Pochhammer symbol (usually denoted in
general as $\left(  a|q\right)  _{n}$) when the parameter $q$ is well defined
and known.

In order to proceed further, we will need some auxiliary results concerning
properties of Pochhammer symbol.

In the sequel, we will often use the following formula, widely known as
Chu-Vandermonde's identity:%
\begin{equation}
\sum_{j=0}^{n}\binom{n}{j}\binom{\alpha}{j}\binom{\beta}{n-j}=\binom
{\alpha+\beta}{n}. \label{CV}%
\end{equation}
It is commonly known that (\ref{CV}), can be derived from the formula below,
known to be a particular case of the Gauss theorem giving the value of
$_{2}F_{1}\left(  a;b,c;1\right)  $ (defined by (\ref{Hyp}, below) for
$a\allowbreak=\allowbreak-n:$
\begin{equation}
\sum_{j=0}^{n}(-1)^{n}\binom{n}{j}\frac{\left(  b\right)  ^{\left(  j\right)
}}{\left(  c\right)  ^{\left(  j\right)  }}=\frac{\left(  c-b\right)
^{\left(  n\right)  }}{\left(  c\right)  ^{\left(  n\right)  }}, \label{Hyp2}%
\end{equation}
that is true for all complex $b$ and $c$ not equal to a negative integer and
integer $n\geq0$.

\section{Pochhammer symbol and Jacobi polynomials\label{Pochh}}

\begin{lemma}
\label{AUX}For all complex $x$ not equal to a non-positive integer and
$n\geq0,$ we have

1)
\begin{equation}
\frac{1}{\left(  x\right)  ^{\left(  n\right)  }}=\frac{1}{(n-1)!}\sum
_{j=0}^{n-1}\binom{n-1}{j}\frac{(-1)^{j}}{x+j}. \label{20}%
\end{equation}

2) Additionally, if $x\neq1,$ we have
\begin{align}
\sum_{j=0}^{n}\left(  -1\right)  ^{j}\binom{n}{j}\frac{(x+2j-1)}{\left(
x+j-1\right)  ^{\left(  n+1\right)  }}  &  =\delta_{n,0},\label{21}\\
\sum_{j=0}^{n}\left(  -1\right)  ^{j}\binom{n}{j}\frac{1}{\left(
x+j-1\right)  ^{\left(  j\right)  }\left(  x+2j\right)  ^{\left(  n-j\right)
}}  &  =\delta_{n,0}. \label{22}%
\end{align}

3) For all complex $x,y$ whose sum is not equal to an integer less or equal to
$1$ and $n\geq0,$ we have
\begin{equation}
1=\sum_{j=0}^{n}\binom{n}{j}\frac{\left(  x+j\right)  ^{\left(  n-j\right)
}\left(  y\right)  ^{\left(  j\right)  }}{\left(  y+x+j-1\right)  ^{\left(
j\right)  }\left(  y+x+2j\right)  ^{\left(  n-j\right)  }}.\label{23}%
\end{equation}

\end{lemma}

\begin{proof}
i) We apply (\ref{Hyp2}) with $n\allowbreak=\allowbreak m-1,$ $c\allowbreak
=\allowbreak x+1,$ $b\allowbreak=\allowbreak x$ and get%
\[
\frac{\left(  1\right)  ^{\left(  m-1\right)  }}{\left(  x+1\right)  ^{\left(
m-1\right)  }}=\sum_{j=0}^{m-1}(-1)^{j}\frac{\left(  x\right)  ^{\left(
j\right)  }}{\left(  x+1\right)  ^{\left(  j\right)  }}=\sum_{j=0}%
^{m-1}(-1)^{j}\frac{x}{\left(  x+j\right)  }.
\]
Now it is enough to notice that $\left(  1\right)  ^{\left(  m-1\right)
}\allowbreak=\allowbreak\left(  m-1\right)  !$ and $x\left(  x+1\right)
^{\left(  m-1\right)  }\allowbreak=\allowbreak\left(  x\right)  ^{\left(
m\right)  }.$

ii) First notice that if $n\allowbreak=\allowbreak0$ then identity is true.
Hence, let us assume that $n\geq1.$ We get using 1) :%
\begin{align*}
&  \sum_{j=0}^{n}\left(  -1\right)  ^{j}\binom{n}{j}\frac{(x+2j-1)}{\left(
x+j-1\right)  ^{\left(  n+1\right)  }}\\
&  =\sum_{j=0}^{n}\left(  -1\right)  ^{j}\binom{n}{j}(x+2j-1)\frac{1}{n!}%
\sum_{k=0}^{n}\binom{n}{k}\frac{\left(  -1\right)  ^{k}}{x-1+j+k}\\
&  =\frac{1}{n!}\sum_{s=0}^{n}\frac{\left(  -1\right)  ^{s}}{x-1+s}\sum
_{j=0}^{s}\left(  x+2j-1\right)  \binom{n}{j}\binom{n}{s-j}\\
&  +\sum_{s=n+1}^{2n}\frac{\left(  -1\right)  ^{s}}{x-1+s}\sum_{j=s-n}%
^{n}\left(  x+2j-1\right)  \binom{n}{j}\binom{n}{s-j}.
\end{align*}

We have further
\begin{gather*}
\sum_{j=0}^{s}\left(  x+2j-1\right)  \binom{n}{j}\binom{n}{s-j}=(x-1)\binom
{2n}{s}+2\sum_{j=0}^{s}j\binom{n}{j}\binom{n}{s-j}\\
=(x-1)\binom{2n}{s}+2n\sum_{k=0}^{s-1}\binom{n-1}{k}\binom{n}{s-1-k}\\
=(x-1)\binom{2n}{s}+2n\binom{2n-1}{s-1}=(x-1+s)\binom{2n}{s}.
\end{gather*}
In the penultimate line of the above calculations, we applied (\ref{CV}). We
also have setting $k\allowbreak=\allowbreak n-j$%
\begin{align*}
&  \sum_{j=s-n}^{n}\left(  x+2j-1\right)  \binom{n}{j}\binom{n}{s-j}\\
&  =\sum_{k=0}^{2n-s}(x+2k-2n-1)\binom{n}{s-n+k}\binom{n}{k}\\
&  =(x-1-2n)\binom{2n}{2n-s}+2\sum_{k=0}^{2n-s}k\binom{n}{s-n+k}\binom{n}{k}\\
&  =(x-1-2n)\binom{2n}{2n-s}+(2n-s)\binom{2n}{2n-s}\\
&  =+(x-1+s)\binom{2n}{s}.
\end{align*}
Again we have applied (\ref{CV}) in the above-mentioned calculations. Hence we
get
\begin{align*}
&  \sum_{j=0}^{n}\left(  -1\right)  ^{j}\binom{n}{j}\frac{(x+2j-1)}{\left(
x+j-1\right)  ^{\left(  n+1\right)  }}\\
&  =\sum_{s=0}^{2n}\frac{(-1)^{s}}{x-1+s}(x-1+s)\binom{2n}{s}=0,
\end{align*}
for $n\geq1.$

2) (\ref{22}) is the variation of (\ref{21}), since
\[
\frac{1}{\left(  x+j-1\right)  ^{\left(  j\right)  }\left(  x+2j\right)
^{\left(  n-j\right)  }}=\frac{(x+2j-1)}{\left(  x+j-1\right)  ^{\left(
n+1\right)  }}.
\]

3) We will apply the identity (\cite{Szab25}(2.9)) with $b\allowbreak
=\allowbreak x+y$ and $a\allowbreak=\allowbreak x$. We get%
\begin{align*}
1  &  =\sum_{s=0}^{n}\binom{n}{s}\frac{\left(  x+s\right)  ^{\left(
n-s\right)  }}{\left(  y+x+s-1\right)  ^{\left(  s\right)  }\left(
y+x+2s\right)  ^{\left(  n-s\right)  }}\\
&  \times\sum_{j=0}^{s}\left(  -1\right)  ^{s-j}\binom{s}{j}\left(
y+x+s-1\right)  ^{\left(  j\right)  }\left(  x+j\right)  ^{\left(  s-j\right)
}\\
&  =\sum_{j=0}^{n}\binom{n}{s}\sum_{s=j}^{n}\left(  -1\right)  ^{s-j}%
\binom{n-j}{s-j}\frac{\left(  x+j\right)  ^{\left(  s-j\right)  }\left(
x+s\right)  ^{\left(  n-s\right)  }\left(  y+x+s-1\right)  ^{\left(  j\right)
}}{\left(  y+x+s-1\right)  ^{\left(  s\right)  }\left(  y+x+2s\right)
^{\left(  n-s\right)  }}%
\end{align*}%
\begin{align*}
&  =\sum_{j=0}^{n}\binom{n}{s}\sum_{m=0}^{n-j}\left(  -1\right)  ^{m}%
\binom{n-j}{m}\times\\
&  \frac{\left(  x+j\right)  ^{\left(  m\right)  }\left(  x+m+j\right)
^{\left(  n-j-m\right)  }}{\left(  y+x+2j+m-1\right)  ^{\left(  m\right)
}\left(  y+x+2j+2m\right)  ^{\left(  n-j-m\right)  }}\\
&  =\sum_{j=0}^{n}\binom{n}{s}\left(  x+j\right)  ^{\left(  n-j\right)  }\\
&  \times\sum_{m=0}^{n-j}\left(  -1\right)  ^{m}\binom{n-j}{m}\frac{1}{\left(
y+x+2j+m-1\right)  ^{\left(  m\right)  }\left(  y+x+2j+2m\right)  ^{\left(
n-j-m\right)  }}.
\end{align*}
Now we apply (\ref{22}).
\end{proof}

\begin{remark}
The proof of (\ref{20}) was suggested by Tom Koornwinder in a private
communication. Similarly, Tom Koornwinder noticed that (\ref{21}) follows the
direct application of Dixon's formula \cite{NIST}(16.4.4) with $a\allowbreak
=\allowbreak x-1,$ $b\allowbreak=\allowbreak-n,$ $c\allowbreak=\allowbreak
(x+1)/2$. He also noticed that (\ref{23}) can be proved by applying the
formula Vol. I, 4.5(4): of \cite{Bat23} with $c\allowbreak=\allowbreak-n,$
$a\allowbreak=\allowbreak x+y-1,$ $b\allowbreak=\allowbreak y$.
\end{remark}

Let us recall the definition of Beta distribution. 

On one hand, we have the distribution with the density:%
\[
f(x|a,b)=%
\begin{cases}
0\text{,} & \text{if }x\notin\lbrack0,1]\text{;}\\
x^{a-1}(1-x)^{b-1}/B(a,b)\text{,} & \text{if }0\leq x\leq1\text{,}%
\end{cases}
\]
where $B(a,b)$ is Euler's beta function, which is defined for all real $a,b$
such that $a,b>0$.

On the other hand, the following function is also called the density of Beta
distributions. It has the following density:%
\[
h(x|a,b)=%
\begin{cases}
(x+1)^{a-1}(1-x)^{b-1}/(B(a,b)2^{a+b-1})\text{,} & \text{if }\left\vert
x\right\vert \leq1\text{;}\\
0\text{,} & \text{if otherwise.}%
\end{cases}
\]
Notice that the above definitions of Beta distribution are somewhat different
from those presented in books on special functions. The definitions of this
paper are probabilistic in origin. The difference is not big but one has to be
aware of it.

It is common knowledge (see, e.g., \cite{Andrews1999}) that the polynomials
that are orthogonal with respect to $h$ are Jacobi polynomials defined by the
formula:%
\begin{equation}
J_{n}(x|a,b)=\frac{1}{n!}\sum_{m=0}^{n}\binom{n}{m}\left(  a+b+n-1\right)
^{\left(  m\right)  }(b+m)^{(n-m)}(x-1)^{m}/2^{m}. \label{J1}%
\end{equation}
Following a simple change of variables under the integral, we deduce, that the
following family of polynomials:%
\begin{equation}
K_{n}(x|a,b)=\frac{1}{n!}\sum_{m=0}^{n}\binom{n}{m}\left(  a+b+n-1\right)
^{\left(  m\right)  }(b+m)^{(n-m)}(x-1)^{m}, \label{Ja}%
\end{equation}
is orthogonal with respect to the distribution with the density $f$.

\begin{remark}
Let's make it clear that polynomials $\left\{  J_{n}\right\}  $ and $\left\{
K_{n}\right\}  $ are defined for all complex values of parameters $a$ and $b$,
but they are orthogonal with respect to positive measure only when they are
actually positive.
\end{remark}

The following two families of rational functions of parameters are important
for our purpose. Let us denote by
\begin{align}
e_{n,m}(a,b)  &  =\binom{n}{m}\left(  a+b+n-1\right)  ^{\left(  m\right)
}(b+m)^{(n-m)}/n!,\label{cnj}\\
\tilde{e}_{n,m}(a,b)  &  =(-1)^{n-m}\frac{n!\left(  b+m\right)  ^{\left(
n-m\right)  }}{(n-m)!\left(  a+b+m-1\right)  ^{\left(  m\right)  }\left(
a+b+2m\right)  ^{\left(  n-m\right)  }}\label{dnj}\\
&  =(-1)^{n-m}\frac{n!\left(  b+m\right)  ^{\left(  n-m\right)  }\left(
a+b+2m-1\right)  }{(n-m)!\left(  a+b+m-1\right)  ^{\left(  n+1\right)  }}.
\label{dnj2}%
\end{align}

Simply, because we have%
\begin{align}
J_{n}(x|a,b)\allowbreak &  =\allowbreak\sum_{m=0}^{n}e_{n,m}(a,b)(x-1)^{m}%
/2^{m},\label{fex1}\\
K_{n}(x|a,b)  &  =\sum_{m=0}^{n}e_{n,m}(a,b)(x-1)^{m}.\label{kex1}\\
\left(  x-1\right)  ^{n}/2^{n}  &  =\sum_{m=0}^{n}\tilde{e}_{n,m}%
(a,b)J_{m}(x|a,b),\label{fex2}\\
\left(  x-1\right)  ^{n}  &  =\sum_{m=0}^{n}\tilde{e}_{n,m}(a,b)K_{m}(x|a,b).
\label{kex2}%
\end{align}

\begin{remark}
Notice that since (\ref{fex1}) and (\ref{fex2}) concern polynomials in
unknowns $x,$ $a,$ $b$ we can extend these identities to all complex values of
these unknowns. The first of these identities is obvious given (\ref{J1}). The
second one is proved either in \cite{IA} (section 4.2) or in \cite{Szab25}
(section 2).
\end{remark}

Following formula (4.1.5) of \cite{IA}, we deduce that polynomials $\left\{
K_{n}\right\}  $ and $\left\{  J_{n}\right\}  $ are not monic. The coefficient
by $x^{n}$ in $K_{n}$ is equal to%
\begin{equation}
\frac{\left(  a+b+n-1\right)  ^{\left(  n\right)  }}{n!2^{n}}. \label{lead}%
\end{equation}

It is also known that for $n\geq1$ we have
\[
\int_{-1}^{1}J_{n}^{2}(x|a,b)h\left(  x|a,b\right)  dx=\frac{\left(  a\right)
^{\left(  n\right)  }\left(  b\right)  ^{\left(  n\right)  }}{n!\left(
a+b+2n-1\right)  \left(  a+b\right)  ^{\left(  n-1\right)  }}.
\]
One can also show, by the simple change of the variable under the integral
that for $n\geq1$ we have
\begin{equation}
\int_{0}^{1}K_{n}^{2}(x|a,b)f\left(  x|a,b\right)  dx=\frac{\left(  a\right)
^{\left(  n\right)  }\left(  b\right)  ^{\left(  n\right)  }}{n!\left(
a+b+2n-1\right)  \left(  a+b\right)  ^{\left(  n-1\right)  }}.\label{SQR}%
\end{equation}

Let us remark that the particular cases of Jacobi polynomials are commonly
used under other names. We recently recalled those specific cases and their
traditional names in \cite{Szab25}.

In the sequel, we will need the following auxiliary result.

\begin{proposition}
For all $n\geq0,$ $a,b>0,$ we have

i)%
\begin{align}
(-1)^{n}J_{n}(-x|a,b)  &  =J_{n}(x|b,a),\label{O1}\\
J_{n}(2x-1|a,b)  &  =K_{n}(x|a,b),\label{O1+}\\
(-1)^{n}K_{n}(x|a,b)  &  =K_{n}(1-x|b,a). \label{O2}%
\end{align}

ii) Let us assume additionally that $c>0$, then we have%
\begin{align*}
\int_{0}^{1}K_{n}(x|a,c)f(x|a,b)dx  &  =\frac{\left(  a\right)  ^{\left(
n\right)  }\left(  c-b\right)  ^{\left(  n\right)  }}{n!\left(  a+b\right)
^{\left(  n\right)  }},\\
\int_{0}^{1}K_{n}(x|c,b)f(x|a,b)dx  &  =(-1)^{n}\frac{\left(  b\right)
^{\left(  n\right)  }\left(  c-a\right)  ^{\left(  n\right)  }}{n!\left(
a+b\right)  ^{\left(  n\right)  }},
\end{align*}

iii) and%
\begin{equation}
K_{n}(0|a,b)=\frac{(-1)^{n}}{n!}\left(  a\right)  ^{\left(  n\right)  }%
,~K_{n}(1|a,b)=\frac{1}{n!}\left(  b\right)  ^{\left(  n\right)  }.
\label{K01}%
\end{equation}

\end{proposition}

\begin{proof}
i) The first of these identities is shown, e.g., in (see \cite{IA} (4.14)).
The second one follows the fact that if $x\in\lbrack0,1],$ then $2x-1\in
\lbrack-1,1]$ and that $2h(2x-1|a,b)\allowbreak=\allowbreak f(x|a,b).$ So by a
suitable change of variables orthogonality of polynomials $\left\{
J_{n}\right\}  $ with respect to $h$ is equivalent to orthogonality of
polynomials $\left\{  K_{n}\right\}  $ with respect to $f$.

Now, let us prove the third one. We have
\begin{gather*}
K_{n}(1-x|b,a)=\frac{1}{n!}\sum_{m=0}^{n}\binom{n}{m}\left(  a+b+n-1\right)
^{\left(  m\right)  }(a+m)^{(n-m)}(1-x-1)^{m}\\
=\frac{1}{n!}\sum_{m=0}^{n}\binom{n}{m}\left(  a+b+n-1\right)  ^{\left(
m\right)  }(a+m)^{(n-m)}(-1)^{m}(x-1+1)^{m}\\
=\frac{1}{n!}\sum_{m=0}^{n}\binom{n}{m}\left(  a+b+n-1\right)  ^{\left(
m\right)  }(a+m)^{(n-m)}(-1)^{m}\sum_{k=0}^{m}\binom{m}{k}(x-1)^{k}%
\end{gather*}%
\begin{gather*}
=\frac{1}{n!}(-1)^{n}\sum_{k=0}^{n}\binom{n}{k}(x-1)^{k}\left(
a+b+n-1\right)  ^{\left(  k\right)  }\\
\times\sum_{m\neq k}^{n}\binom{n-k}{m-k}(-1)^{n-m}\left(  a+b+n-1+k\right)
^{\left(  m-k\right)  }(a+m)^{(n-k-(m-k))}%
\end{gather*}%
\begin{gather*}
=\frac{1}{n!}(-1)^{n}\sum_{k=0}^{n}\binom{n}{k}(x-1)^{k}\left(
a+b+n-1\right)  ^{\left(  k\right)  }\\
\times\sum_{s=0}^{n-k}\binom{n-k}{s}(-1)^{n-k-s}\left(  a+k+b+k+n-k-1\right)
^{\left(  s\right)  }(a+k+s)^{(n-k-s)}%
\end{gather*}%
\[
=\frac{1}{n!}(-1)^{n}\sum_{k=0}^{n}\binom{n}{k}(x-1)^{k}\left(
a+b+n-1\right)  ^{\left(  k\right)  }\left(  b+k\right)  ^{\left(  n-k\right)
}=\left(  -1\right)  ^{n}K_{n}(x|a,b).
\]

In the last line, we applied formula (2.8) of \cite{Szab25}.

ii) First, let us calculate
\begin{align*}
M_{k}(a,b)  &  =\int_{0}^{1}(x-1)^{k}f\left(  x|a,b\right)  dx=\frac{\left(
-1\right)  ^{k}}{B(a,b)}\int_{0}^{1}x^{a-1}(1-x)^{k+b-1}dx\\
&  =\frac{\left(  -1\right)  ^{k}B(a,k+b)}{B(a,b)}=(-1)^{k}\frac
{\Gamma(a)\Gamma(b+k)\Gamma(a+b)}{\Gamma(a+b+k)\Gamma(a)\Gamma(b)}=\left(
-1\right)  ^{k}\frac{\left(  b\right)  ^{\left(  k\right)  }}{\left(
a+b\right)  ^{\left(  k\right)  }}.
\end{align*}

Now, instead of showing that
\begin{gather}
\frac{\left(  a\right)  ^{\left(  n\right)  }\left(  c-b\right)  ^{\left(
n\right)  }}{n!\left(  a+b\right)  ^{\left(  n\right)  }}=\int_{0}^{1}%
K_{n}(x|a,c)f(x|a,b)dx\label{dN1}\\
=\frac{1}{n!}\sum_{k=0}^{n}\binom{n}{k}\left(  a+c+n-1\right)  ^{\left(
k\right)  }(c+k)^{(n-k)}M_{k}(a,b).\nonumber
\end{gather}
We will use the fact that
\[
M_{n}(a,b)=\sum_{j=0}^{n}\tilde{e}_{n,j}(a,c)\int_{0}^{1}K_{j}%
(x|a,c)f(x|a,b)dx.
\]
In other words, we will show that
\begin{align*}
\left(  -1\right)  ^{n}\frac{\left(  b\right)  ^{\left(  n\right)  }}{\left(
a+b\right)  ^{\left(  n\right)  }}  &  =\sum_{m=0}^{n}(-1)^{n-m}%
\frac{n!\left(  b+m\right)  ^{\left(  n-m\right)  }}{(n-m)!\left(
a+b+m-1\right)  ^{\left(  m\right)  }\left(  a+b+2m\right)  ^{\left(
n-m\right)  }}\\
&  \times\frac{\left(  a\right)  ^{\left(  m\right)  }\left(  c-b\right)
^{\left(  m\right)  }}{m!\left(  a+b\right)  ^{\left(  m\right)  }}.
\end{align*}
After cancelling out $(-1)^{n}$ and multiplying both sides by $\left(
a+b\right)  ^{\left(  n\right)  }$, we have to show that :%
\[
\left(  b\right)  ^{\left(  n\right)  }=\sum_{m=0}^{n}(-1)^{m}\binom{n}%
{m}\frac{\left(  c+m\right)  ^{\left(  n-m\right)  }\left(  a\right)
^{\left(  m\right)  }\left(  c-b\right)  ^{\left(  m\right)  }\left(
a+b+m\right)  ^{\left(  n-m\right)  }}{\left(  a+c+m-1\right)  ^{\left(
m\right)  }\left(  a+c+2m\right)  ^{\left(  n-m\right)  }}.
\]

Now, we apply the following formula (2.9) from \cite{Szab25} (Lemma 2) applied
with $b\allowbreak=\allowbreak a+c$ and $a\allowbreak=\allowbreak a+b,$ so
that $b-a\allowbreak=\allowbreak c-b$.

We have further:%

\begin{align*}
\left(  b\right)  ^{\left(  n\right)  }  &  =\sum_{m=0}^{n}(-1)^{m}\binom
{n}{m}\frac{\left(  c+m\right)  ^{\left(  n-m\right)  }\left(  a\right)
^{\left(  m\right)  }\left(  c-b\right)  ^{\left(  m\right)  }\left(
a+b+m\right)  ^{\left(  n-m\right)  }}{\left(  a+c+m-1\right)  ^{\left(
m\right)  }\left(  a+c+2m\right)  ^{\left(  n-m\right)  }}\\
&  \times\sum_{j-0}^{m}\left(  -1\right)  ^{m-j}\binom{m}{j}\left(
a+c+m-1\right)  ^{\left(  j\right)  }\left(  a+b+j\right)  ^{\left(
m-j\right)  }\\
&  =\sum_{j=0}^{n}\left(  -1\right)  ^{j}\binom{n}{j}\left(  a+b+j\right)
^{\left(  n-j\right)  }\left(  a\right)  ^{\left(  j\right)  }\\
&  \times\sum_{m=j}^{n}\binom{n-j}{m-j}\frac{\left(  c+m\right)  ^{\left(
n-m\right)  }\left(  a+j\right)  ^{\left(  m-j\right)  }\left(
a+c+m-1\right)  ^{\left(  j\right)  }}{\left(  a+c+m-1\right)  ^{\left(
m\right)  }\left(  a+c+2m\right)  ^{\left(  n-m\right)  }}.
\end{align*}
After denoting $s\allowbreak=\allowbreak m-j,$ we have%
\begin{align*}
\left(  b\right)  ^{\left(  n\right)  }  &  =\sum_{j=0}^{n}\left(  -1\right)
^{j}\binom{n}{j}\left(  a+b+j\right)  ^{\left(  n-j\right)  }\left(  a\right)
^{\left(  j\right)  }\\
&  \times\sum_{s=1}^{n-j}\binom{n-j}{s}\frac{\left(  c+j+s\right)  ^{\left(
n-j-s\right)  }\left(  a+j\right)  ^{\left(  s\right)  }}{\left(
a+c+2j+s-1\right)  ^{\left(  s\right)  }\left(  a+c+2j+2s\right)  ^{\left(
n-j-s\right)  }}.
\end{align*}
But, following Lemma \ref{AUX} (3)) with $x\allowbreak=\allowbreak c+j,$
$y\allowbreak=\allowbreak a+j,$ we deduce that
\[
\sum_{s=1}^{n-j}\binom{n-j}{s}\frac{\left(  c+j+s\right)  ^{\left(
n-j-s\right)  }\left(  a+j\right)  ^{\left(  s\right)  }}{\left(
a+c+2j+s-1\right)  ^{\left(  s\right)  }\left(  a+c+2j+2s\right)  ^{\left(
n-j-s\right)  }}=1.
\]
Further, we apply formula (2.8) from \cite{Szab25}.

iii) The value at $1$ is obvious, following (\ref{kex1}). For the value at $0$
we have apply (\ref{O2}).
\end{proof}

The following identities are received immediately as a corollary.

\begin{corollary}
For all complex $a,$ $b,$ $c$ such that and $n\geq0$ we have%
\begin{equation}
\frac{\left(  a\right)  ^{\left(  n\right)  }\left(  c-b\right)  ^{\left(
n\right)  }}{\left(  a+b\right)  ^{\left(  n\right)  }}=\sum_{k=0}^{n}\left(
-1\right)  ^{k}\binom{n}{k}\frac{\left(  a+c+n-1\right)  ^{\left(  k\right)
}(c+k)^{(n-k)}\left(  b\right)  ^{\left(  k\right)  }}{\left(  a+b\right)
^{\left(  k\right)  }}, \label{11}%
\end{equation}

\begin{equation}
\frac{\left(  a\right)  ^{\left(  n\right)  }\left(  b\right)  ^{\left(
n\right)  }}{\left(  c\right)  ^{\left(  n\right)  }}=\sum_{k=0}^{n}\left(
-1\right)  ^{k}\binom{n}{k}\frac{\left(  b+c+n-1\right)  ^{\left(  k\right)
}\left(  b+c-a+k\right)  ^{\left(  n-k\right)  }\left(  c-a\right)  ^{\left(
k\right)  }}{\left(  c\right)  ^{\left(  k\right)  }}, \label{12}%
\end{equation}

\begin{align}
\left(  a\right)  ^{\left(  n\right)  }\left(  c-b\right)  ^{\left(  n\right)
}  &  =\sum_{k=0}^{n}\left(  -1\right)  ^{k}\binom{n}{k}\left(
a+c+n-1\right)  ^{\left(  k\right)  }\label{13}\\
&  \times(c+k)^{(n-k)}\left(  b\right)  ^{\left(  k\right)  }\left(
a+b+k\right)  ^{\left(  n-k\right)  },\nonumber
\end{align}

\begin{align}
\left(  a\right)  ^{\left(  n\right)  }\left(  b\right)  ^{\left(  n\right)
}  &  =\sum_{k=0}^{n}\left(  -1\right)  ^{k}\binom{n}{k}\left(
a+b+c+n-1\right)  ^{\left(  k\right)  }\label{I4}\\
&  \times\left(  c+b+k\right)  ^{\left(  n-k\right)  }\left(  c\right)
^{\left(  k\right)  }\left(  a+c+k\right)  ^{\left(  n-k\right)  }.\nonumber
\end{align}

In particular, we have
\begin{align}
\left(  a+b-1\right)  \sum_{j=0}^{n}\left(  -1\right)  ^{j}\binom{n}{j}\left(
a+b+j\right)  ^{\left(  n-1\right)  }  &  =\delta_{n,0},\label{001}\\
\sum_{j=0}^{n}\left(  -1\right)  ^{j}\binom{n}{j}\left(  a+b+j\right)
^{\left(  n-j\right)  }\left(  a+b+n-1\right)  ^{\left(  j\right)  }  &
=\delta_{n,0}, \label{002}%
\end{align}

\end{corollary}

\begin{proof}
(\ref{11}) is nothing else but (\ref{dN1}), which was proved indirectly,
above. (\ref{12}) is (\ref{11}) with $c-b$ replaced by $b$ while $a+b$ is
replaced by $c.$ (\ref{13}) is (\ref{11}) multiplied by $\left(  a+b\right)
^{\left(  n\right)  }$ on both sides. (\ref{I4}) is obtained from (\ref{12})
by the similar multiplication. (\ref{001}) is obtained from (\ref{11}) by
setting $c\allowbreak=\allowbreak b$ and noticing that $(b+k)^{(n-k)}\left(
b\right)  ^{\left(  k\right)  }\allowbreak=\allowbreak\left(  b\right)
^{\left(  n\right)  }$ and can be put outside the sum. Further we notice that
\begin{align*}
\frac{\left(  a+b+n-1\right)  ^{\left(  j\right)  }}{\left(  a+b\right)
^{\left(  j\right)  }}  &  =\frac{\left(  a+b\right)  ^{\left(  n-1+j\right)
}}{\left(  a+b\right)  ^{\left(  n-1\right)  }\left(  a+b\right)  ^{\left(
j\right)  }}\\
&  =\frac{\left(  a+b+j\right)  ^{\left(  n-1\right)  }}{\left(  a+b\right)
^{\left(  n-1\right)  }},
\end{align*}
and we put $1/\left(  a+b\right)  ^{\left(  n-1\right)  }$ outside the sum.
(\ref{002}) is obtained from (\ref{13}) by setting $c=b$.
\end{proof}

\begin{remark}
Again, Tom Koornwinder, in a private communication, noticed that (\ref{11})
can be obtained from the so-called Pfaff-Saalsch\"{u}tz formula \newline(
http://dlmf.nist.gov/16.4.E3) applied to the generalized hypergeometric
function $_{3}F_{2}\left(
\begin{array}
[c]{ccc}%
-n & a+c+n-1 & b\\
c & a+b &
\end{array}
;1\right)  .$ For the definition of hypergeometric function, see the section
\ref{HYP}.
\end{remark}

\section{Hypergeometric function\label{HYP}}

Now, the third object analysed in  the paper is the so-called generalized
hypergeometric function defined by the following infinite series%
\begin{equation}
_{n}F_{m}\left(
\begin{array}
[c]{ccc}%
a_{1} & \ldots & a_{n}\\
b_{1} & \ldots & b_{m}%
\end{array}
;x\right)  =\sum_{j\geq0}\frac{\prod_{k=1}^{n}\left(  a_{k}\right)  ^{\left(
j\right)  }}{\prod_{k=1}^{m}\left(  b_{k}\right)  ^{\left(  j\right)  }}%
\frac{x^{j}}{j!}.\label{GHF}%
\end{equation}
Although one can formally consider this function for all non-negative integers
$n$ and $m,$ we will consider it for all non-negative integers $n\leq m+1.$
This is so because for $n>m+1$ there appear some convergence problems of the
defining series (\ref{GHF}). For details, see \cite{KLS} and \cite{NIST}.
Following \cite{NIST}, 16.2(iii) we will assume that when $n\leq m$ then all
$b_{1},\ldots,b_{m}$ have positive real parts positive reals, when $n=m+1,$
then we will assume that $\operatorname{Re}(\sum_{j=1}^{m}b_{j}-\sum_{j=1}%
^{n}a_{j})>0$ and of course we assume that $\left\vert x\right\vert \leq1$. In
general let us notice that whenever $\prod_{k=1}^{n}a_{k}\allowbreak
=\allowbreak0,$ then $_{n}F_{m}\left(
\begin{array}
[c]{ccc}%
a_{1} & \ldots & a_{n}\\
b_{1} & \ldots & b_{m}%
\end{array}
;x\right)  \allowbreak=\allowbreak1$ for all $x.$ Also when one of the
parameters $a_{1},\ldots,a_{n}$ is equal to a negative integer, say equal to
$-n$, then $_{n}F_{m}\left(
\begin{array}
[c]{ccc}%
a_{1} & \ldots & a_{n}\\
b_{1} & \ldots & b_{m}%
\end{array}
;x\right)  $ is a polynomial in $x$ of order $n$.

Let us remark that the case $n\allowbreak=\allowbreak2$ and $m\allowbreak
=\allowbreak1$ is treated specially. The function $_{2}F_{1}\left(
\begin{array}
[c]{cc}%
a & b\\
c &
\end{array}
;x\right)  $ is called the hypergeometric function, It was studied already by
L. Euler, but its many major properties were obtained by Gauss. It is will be
alternatively denoted as $_{2}F_{1}\left(  a,b;c;x\right)  .$ Hence, we
$_{2}F_{1}$ is the function defined by the following infinite series%
\begin{equation}
_{2}F_{1}\left(
\begin{array}
[c]{cc}%
a & b\\
c &
\end{array}
;x\right)  =~_{2}F_{1}\left(  a,b;c;x\right)  =\sum_{n\geq0}\frac{\left(
a\right)  ^{\left(  n\right)  }\left(  b\right)  ^{\left(  n\right)  }%
}{\left(  c\right)  ^{\left(  n\right)  }}\frac{x^{n}}{n!},\label{Hyp}%
\end{equation}
considered for all complex $\left\vert x\right\vert <1,$ such that $c$ is not
a non-positive integer. If $\operatorname{Re}(c-a-b)>0$ then one extends the
domain of $x$ by adding the unit circle , i.e., when $\left\vert x\right\vert
\allowbreak=\allowbreak1$. To avoid unnecessary complications and also take
into account that we intend to expand the hypergeometric function in the
orthogonal series, we will always assume that parameters $a,$ $b,$ $c$ are
real and such that $c-a-b>0$. After obtaining the orthogonal expansion, we can
extend it from the real segment $[0,1]$ to the unit circle on the complex plane.

Let us also notice that for all positive $m,a_{1},\ldots,a_{m}$ we have%
\[
_{m+1}F_{m}\left(
\begin{array}
[c]{cccc}%
a_{1} & \ldots & a_{m} & b\\
a_{1} & \ldots & a_{m} &
\end{array}
;x\right)  =(1-x)^{-b}.
\]
Hence, we will also assume that $a\neq c\neq b.$

Again, it is common knowledge that one can present Jacobi polynomials as
certain values of hypergeometric function. Namely, we have
\begin{align}
\frac{\left(  b\right)  ^{\left(  n\right)  }}{n!}~_{2}F_{1}\left(
\begin{array}
[c]{cc}%
-n & a+b+n-1\\
b &
\end{array}
;\frac{1-x}{2}\right)   &  =J_{n}(x|a,b),\label{FJ}\\
\frac{\left(  b\right)  ^{\left(  n\right)  }}{n!}~_{2}F_{1}\left(
\begin{array}
[c]{cc}%
-n & a+b+n-1\\
b &
\end{array}
;x\right)   &  =K_{n}(1-x|a,b),\label{FK1}\\
\frac{\left(  a\right)  ^{\left(  n\right)  }}{n!}~_{2}F_{1}\left(
\begin{array}
[c]{cc}%
-n & a+b+n-1\\
a &
\end{array}
;x\right)   &  =(-1)^{n}K_{n}(x|a,b).\label{FK2}%
\end{align}
Again, since our definition of Beta distribution is slightly different from
the ones found in books on special functions, the above identities take a
slightly different form from the usual ones.

Let us start with the following result. $\operatorname{Re}$

\begin{theorem}
\label{GHyp}For any complex $a_{1},\ldots,a_{n},$ $\operatorname{Re}%
(b_{1})>0,\ldots\operatorname{Re}(b_{m})>0$ and $\operatorname{Re}(\sum
_{j=1}^{m}b_{j}-\sum_{j=1}^{n}a_{j})>0$ and $\left\vert x\right\vert \leq1$
and positive reals $f,d$ we have:

i)%
\begin{gather}
_{n}F_{m}\left(
\begin{array}
[c]{ccc}%
a_{1} & \ldots & a_{n}\\
b_{1} & \ldots & b_{m}%
\end{array}
;(1+x)/2\right)  =\label{GHJ1}\\
\sum_{k\geq0}J_{k}(x|f,d)\frac{\prod_{s=1}^{n}\left(  a_{s}\right)  ^{\left(
k\right)  }}{\prod_{s=1}^{m}\left(  b_{s}\right)  ^{\left(  k\right)  }\left(
d+f+k-1\right)  ^{\left(  k\right)  }}\nonumber\\
\times~_{n+1}F_{m+1}\left(
\begin{array}
[c]{cccc}%
a_{1}+k & \ldots & a_{n}+k & f+k\\
b_{1}+k & \ldots & b_{m}+k & \left(  d+f+2k\right)
\end{array}
;1\right) \nonumber
\end{gather}

ii)
\begin{gather}
_{n}F_{m}\left(
\begin{array}
[c]{ccc}%
a_{1} & \ldots & a_{n}\\
b_{1} & \ldots & b_{m}%
\end{array}
;x\right)  =\label{GHK1}\\
\sum_{j\geq0}K_{j}(x|f,d)\frac{\prod_{s=1}^{n}\left(  a_{s}\right)  ^{\left(
j\right)  }}{\prod_{s=1}^{m}\left(  b_{s}\right)  ^{\left(  j\right)  }\left(
d+f+j-1\right)  ^{\left(  j\right)  }}\nonumber\\
\times~_{n+1}F_{m+1}\left(
\begin{array}
[c]{cccc}%
a_{1}+j & \ldots & a_{n}+j & f+j\\
b_{1}+j & \ldots & b_{m}+j & \left(  d+f+2j\right)
\end{array}
;1\right)  .\nonumber
\end{gather}

iii) Assume additionally that $b_{1}\allowbreak\neq$\allowbreak$1,$ then
\begin{gather}
_{n}F_{m}\left(
\begin{array}
[c]{ccc}%
a_{1} & \ldots & a_{n}\\
b_{1} & \ldots & b_{m}%
\end{array}
;x\right)  =\label{GHS1}\\
(b_{1}-1)\sum_{k\geq0}\frac{(-1)^{k}}{b_{1}-1+k}\frac{x^{k}}{k!}\frac
{\prod_{s=1}^{n}\left(  a_{s}\right)  ^{\left(  k\right)  }}{k!\prod_{s=2}%
^{m}\left(  b_{s}\right)  ^{\left(  k\right)  }}~_{n}F_{m}\left(
\begin{array}
[c]{ccc}%
a_{1}+k & \ldots & a_{n}+k\\
1+k & \ldots & b_{m}+k
\end{array}
;x\right) \nonumber
\end{gather}

\end{theorem}

\begin{proof}
i) We will utilize (\ref{fex2}), (\ref{dnj}) and (\ref{O1}). We get then
\begin{gather*}
_{n}F_{m}\left(
\begin{array}
[c]{ccc}%
a_{1} & \ldots & a_{n}\\
b_{1} & \ldots & b_{m}%
\end{array}
;(1-x)/2\right)  =\sum_{j\geq0}\frac{\prod_{s=1}^{n}\left(  a_{s}\right)
^{\left(  j\right)  }}{\prod_{s=1}^{m}\left(  b_{s}\right)  ^{\left(
j\right)  }}\frac{(1-x)^{j}}{j!2^{j}}\\
=\sum_{j\geq0}(-1)^{j}\frac{\prod_{s=1}^{n}\left(  a_{s}\right)  ^{\left(
j\right)  }}{\prod_{s=1}^{m}\left(  b_{s}\right)  ^{\left(  j\right)  }}%
\frac{1}{j!}\sum_{k=0}^{j}\tilde{e}_{j,k}(d,f)J_{k}(x|d,f)\\
=\sum_{k\geq0}\left(  -1\right)  ^{k}J_{k}(x|d,f)\sum_{j\geq k}(-1\sum
)^{j-k}\frac{\prod_{s=1}^{n}\left(  a_{s}\right)  ^{\left(  j\right)  }}%
{\prod_{s=1}^{m}\left(  b_{s}\right)  ^{\left(  j\right)  }}\\
\times\frac{1}{j!}\frac{(-1)^{j-k}j!\left(  f+k\right)  ^{(j-k)}}{\left(
j-k\right)  !\left(  d+f+k-1\right)  ^{\left(  k\right)  }\left(
d+f+2k\right)  ^{\left(  j-k\right)  }}\\
=\sum_{k\geq0}J_{k}(-x|f,d)\frac{\prod_{s=1}^{n}\left(  a_{s}\right)
^{\left(  k\right)  }}{\prod_{s=1}^{m}\left(  b_{s}\right)  ^{\left(
k\right)  }\left(  d+f+k-1\right)  ^{\left(  k\right)  }}\\
\times\sum_{j\geq k}\frac{\prod_{s=1}^{n}\left(  a_{s}+k\right)  ^{\left(
j-k\right)  }}{\prod_{s=1}^{m}\left(  b_{s}+k\right)  ^{\left(  j-k\right)  }%
}\frac{\left(  f+k\right)  ^{(j-k)}}{\left(  j-k\right)  !\left(
d+f+2k\right)  ^{\left(  j-k\right)  }}\\
=\sum_{k\geq0}J_{k}(-x|f,d)\frac{\prod_{s=1}^{n}\left(  a_{s}\right)
^{\left(  k\right)  }}{\prod_{s=1}^{m}\left(  b_{s}\right)  ^{\left(
k\right)  }\left(  d+f+k-1\right)  ^{\left(  k\right)  }}\\
\times~_{n+1}F_{m+1}\left(
\begin{array}
[c]{cccc}%
a_{1+}+k & \ldots & a_{n}+k & f+k\\
b_{1}+k & \ldots & b_{m}+k & \left(  d+f+2k\right)
\end{array}
;1\right)
\end{gather*}

ii) We proceed in the similar way by this time we will use polynomials $K_{j}%
$, hence we will use (\ref{kex2}) and (\ref{O2}) and in the last stage we
change $1-x$ to $x$.%

\begin{gather*}
_{n}F_{m}\left(
\begin{array}
[c]{ccc}%
a_{1} & \ldots & a_{n}\\
b_{1} & \ldots & b_{m}%
\end{array}
;(1-x)\right)  =\sum_{k\geq0}(-1)^{k}\frac{\prod_{s=1}^{n}\left(
a_{s}\right)  ^{\left(  k\right)  }}{\prod_{s=1}^{m}\left(  b_{s}\right)
^{\left(  k\right)  }}\frac{(x-1)^{k}}{k!}\\
=\sum_{k\geq0}(-1)^{k}\frac{\prod_{s=1}^{n}\left(  a_{s}\right)  ^{\left(
k\right)  }}{\prod_{s=1}^{m}\left(  b_{s}\right)  ^{\left(  k\right)  }}%
\frac{1}{k!}\sum_{j=0}^{k}\tilde{e}_{k,j}(d,f)K_{j}(x|d,f)\\
=\sum_{j\geq0}(-1)^{j}K_{j}(x|d,f)\sum_{k\geq j}\frac{(-1)^{k-j}}{k!}\\
\times\frac{\prod_{s=1}^{n}\left(  a_{s}\right)  ^{\left(  k\right)  }}%
{\prod_{s=1}^{m}\left(  b_{s}\right)  ^{\left(  k\right)  }}\frac
{(-1)^{k-j}k!\left(  b+j\right)  ^{\left(  k-m\right)  }}{(k-j)!\left(
a+b+m-1\right)  ^{\left(  j\right)  }\left(  a+b+2m\right)  ^{\left(
k-j\right)  }}.
\end{gather*}

iii) We use (\ref{20}) and get
\begin{gather*}
_{n}F_{m}\left(
\begin{array}
[c]{ccc}%
a_{1} & \ldots & a_{n}\\
b_{1} & \ldots & b_{m}%
\end{array}
;x\right)  =(b_{1}-1)\sum_{j\geq0}\frac{\prod_{s=1}^{n}\left(  a_{s}\right)
^{\left(  j\right)  }}{\prod_{s=2}^{m}\left(  b_{s}\right)  ^{\left(
j\right)  }}\frac{x^{j}}{j!}\frac{1}{j!}\sum_{k=0}^{j}\binom{j}{k}%
\frac{(-1)^{k}}{b_{1}-1+k}\\
=(b_{1}-1)\sum_{k\geq0}\frac{(-1)^{k}}{b_{1}-1+k}\frac{x^{k}}{k!}\sum_{j\geq
k}\frac{x^{j-k}}{(j-k)!}\frac{\prod_{s=1}^{n}\left(  a_{s}\right)  ^{\left(
j\right)  }}{(1)^{\left(  j\right)  }\prod_{s=2}^{m}\left(  b_{s}\right)
^{\left(  j\right)  }}\\
=(b_{1}-1)\sum_{k\geq0}\frac{(-1)^{k}}{b_{1}-1+k}\frac{x^{k}}{k!}\frac
{\prod_{s=1}^{n}\left(  a_{s}\right)  ^{\left(  k\right)  }}{\prod_{s=2}%
^{m}\left(  b_{s}\right)  ^{\left(  k\right)  }}\sum_{j\geq k}\frac{x^{j-k}%
}{(j-k)!}\frac{\prod_{s=1}^{n}\left(  a_{s}+k\right)  ^{\left(  j-k\right)  }%
}{(k+1)^{\left(  j-k\right)  }\prod_{s=2}^{m}\left(  b_{s}+k\right)  ^{\left(
j-k\right)  }}\\
=(b_{1}-1)\sum_{k\geq0}\frac{(-1)^{k}}{b_{1}-1+k}\frac{x^{k}}{k!}\frac
{\prod_{s=1}^{n}\left(  a_{s}\right)  ^{\left(  k\right)  }}{k!\prod_{s=2}%
^{m}\left(  b_{s}\right)  ^{\left(  k\right)  }}~_{n}F_{m}\left(
\begin{array}
[c]{ccc}%
a_{1}+k & \ldots & a_{n}+k\\
1+k & \ldots & b_{m}+k
\end{array}
;x\right)
\end{gather*}

\end{proof}

\begin{corollary}
[Hypergeometric case]1) For all complex $x,$ $a,$ $b,$ $c$ , and real $d$ ,
$f$ such that $\operatorname{Re}(c-a-b)>0,$ $\left\vert x\right\vert \leq1,$
$d>0,$ $f>0,$ we have the following expansions of the hypergeometric function
in orthogonal series%
\begin{gather}
_{2}F_{1}\left(  a,b;c;\frac{\left(  1+x\right)  }{2}\right)  =\sum_{j\geq
0}J_{j}\left(  x|f,d\right)  \frac{\left(  a\right)  ^{\left(  j\right)
}\left(  b\right)  ^{\left(  j\right)  }}{\left(  c\right)  ^{\left(
j\right)  }\left(  d+f+j-1\right)  ^{\left(  j\right)  }}\times\label{2F1-1}\\
_{3}F_{2}\left(
\begin{array}
[c]{ccc}%
a+j & b+j & f+j\\
c+j & d+f+2j &
\end{array}
;1\right)  =\frac{\Gamma(d+f)\Gamma(c+d-a-b)}{\Gamma(d+f-a)\Gamma
(c+d-b)}\times\nonumber\\
\sum_{j\geq0}J_{j}\left(  x|f,d\right)  \frac{\left(  a\right)  ^{\left(
j\right)  }\left(  b\right)  ^{\left(  j\right)  }\left(  d+f\right)
^{\left(  2j\right)  }}{\left(  c\right)  ^{\left(  j\right)  }\left(
d+f+j-1\right)  ^{\left(  j\right)  }\left(  d+f-a\right)  ^{\left(  j\right)
}\left(  c+d-b\right)  ^{\left(  j\right)  }}\nonumber\\
\times~_{3}F_{2}\left(
\begin{array}
[c]{ccc}%
a+j & c-b & c-f\\
c+j & c+d-b+j &
\end{array}
;1\right)  .\nonumber
\end{gather}

\begin{gather}
_{2}F_{1}\left(  a,b;c;x\right)  =\sum_{j\geq0}K_{j}\left(  x|f,d\right)
\frac{\left(  a\right)  ^{\left(  j\right)  }\left(  b\right)  ^{\left(
j\right)  }}{\left(  c\right)  ^{\left(  j\right)  }\left(  d+f+j-1\right)
^{\left(  j\right)  }}\label{2F1-2}\\
\times~_{3}F_{2}\left(
\begin{array}
[c]{ccc}%
a+j & b+j & f+j\\
c+j & d+f+2j &
\end{array}
;1\right)  =\frac{\Gamma(d+f)\Gamma(c+d-a-b)}{\Gamma(d+f-a)\Gamma
(c+d-b)}\nonumber\\
\sum_{j\geq0}K_{j}\left(  x|f,d\right)  \frac{\left(  a\right)  ^{\left(
j\right)  }\left(  b\right)  ^{\left(  j\right)  }\left(  d+f\right)
^{\left(  2j\right)  }}{\left(  c\right)  ^{\left(  j\right)  }\left(
d+f+j-1\right)  ^{\left(  j\right)  }\left(  d+f-a\right)  ^{\left(  j\right)
}\left(  c+d-b\right)  ^{\left(  j\right)  }}\nonumber\\
\times~_{3}F_{2}\left(
\begin{array}
[c]{ccc}%
a+j & c-b & c-f\\
c+j & c+d-b+j &
\end{array}
;1\right)  .\nonumber
\end{gather}

2) If $c\neq1,$ $\operatorname{Re}(c-a-b)>0$ and $\left\vert x\right\vert <1,$
then we have also:%
\begin{gather}
_{2}F_{1}\left(  a,b;c;x\right)  =\left(  c-1\right)  \sum_{j\geq0}%
\frac{(-1)^{j}\left(  a\right)  ^{\left(  j\right)  }\left(  b\right)
^{\left(  j\right)  }x^{j}}{\left(  c-1+j\right)  \left(  j!\right)  ^{2}%
}~_{2}F_{1}\left(  a+j,b+j;j+1;x\right) \label{2F1-3}\\
=(1-x)^{1-a-b}\left(  c-1\right) \nonumber\\
\times\sum_{j\geq0}\frac{(-1)^{j}\left(  a\right)  ^{\left(  j\right)
}\left(  b\right)  ^{\left(  j\right)  }x^{j}}{\left(  c-1+j\right)  \left(
j!\right)  ^{2}(1-x)^{j}}~_{2}F_{1}\left(  1-a,1-b;j+1;x\right)  .\nonumber
\end{gather}

\end{corollary}

\begin{proof}
We apply assertions of the Theorem \ref{GHyp} for $n\allowbreak=\allowbreak2$
and $m\allowbreak=\allowbreak1$ with $a_{1}\allowbreak=\allowbreak a,$
$a_{2}\allowbreak=\allowbreak b,$ $b_{1}\allowbreak=\allowbreak c.$ Now to
transform $~_{3}F_{2}\left(
\begin{array}
[c]{ccc}%
a+j & b+j & f+j\\
c+j & d+f+2j &
\end{array}
;1\right)  $ we apply (16.4.11) of \cite{NIST} and (\ref{Gn}).

2) To get the last equality, we applied the well-known Euler transformation.
\end{proof}

\begin{remark}
Let us notice that we can use identity (16.4.11) of \cite{NIST} concerning
$_{3}F_{2}\left(
\begin{array}
[c]{ccc}%
a & b & c\\
d & f &
\end{array}
;1\right)  $ in another way and get another forms of the basic expansions
(\ref{2F1-1}) and (\ref{2F1-2}). We get then%
\begin{gather}
_{2}F_{1}\left(  a,b;c;\frac{\left(  1+x\right)  }{2}\right)  =\frac
{\Gamma(c)\Gamma(c+d-a-b)}{\Gamma(c-a)\Gamma(c+d-b)}\sum_{j\geq0}J_{j}\left(
x|f,d\right) \label{2F1-1a}\\
\times\frac{\left(  a\right)  ^{\left(  j\right)  }\left(  b\right)  ^{\left(
j\right)  }}{\left(  d+f+j-1\right)  ^{\left(  j\right)  }(c+d-b)^{\left(
j\right)  }}\times\\
~_{3}F_{2}\left(
\begin{array}
[c]{ccc}%
a+j & d+f-b+j & d+j\\
d+c-b+j & d+f+2j &
\end{array}
;1\right) \nonumber\\
=\frac{\Gamma(c)\Gamma(c+d-a-b)}{\Gamma(c-f)\Gamma(c+d+f-a-b)}\sum_{j\geq
0}J_{j}\left(  x|f,d\right) \label{2F1-1b}\\
\times\frac{\left(  a\right)  ^{\left(  j\right)  }\left(  b\right)  ^{\left(
j\right)  }}{\left(  d+f+j-1\right)  ^{\left(  j\right)  }(c+d+f-a-b)^{\left(
j\right)  }}\\
\times~_{3}F_{2}\left(
\begin{array}
[c]{ccc}%
f+j & d+f-b+j & d+f-a+j\\
c+d+f-a-b+j & d+f+2j &
\end{array}
;1\right) \nonumber
\end{gather}%
\begin{gather}
=\frac{\Gamma(d+f)\Gamma(c+d-a-b)}{\Gamma(d)\Gamma(c+d+f-a-b)}\sum_{j\geq
0}J_{j}\left(  x|f,d\right)  \frac{\left(  a\right)  ^{\left(  j\right)
}\left(  b\right)  ^{\left(  j\right)  }}{\left(  c\right)  ^{\left(
j\right)  }\left(  d+f+j-1\right)  ^{\left(  j\right)  }}\times\label{2F1-1d}%
\\
\frac{(d+f)^{\left(  2j\right)  }}{(d)^{\left(  j\right)  }%
(c+d+f-a-b)^{\left(  j\right)  }}~_{3}F_{2}\left(
\begin{array}
[c]{ccc}%
f+j & c-a & c-b\\
c+j & c+d+f-a-b+j &
\end{array}
;1\right) \nonumber
\end{gather}

and
\begin{gather}
_{2}F_{1}\left(  a,b;c;x\right)  =\frac{\Gamma(c)\Gamma(c+d-a-b)}%
{\Gamma(c-a)\Gamma(c+d-b)}\sum_{j\geq0}K_{j}\left(  x|f,d\right)
\times\label{2F1-2a}\\
\frac{\left(  a\right)  ^{\left(  j\right)  }\left(  b\right)  ^{\left(
j\right)  }}{\left(  d+f+j-1\right)  ^{\left(  j\right)  }(c+d-b)^{\left(
j\right)  }}~_{3}F_{2}\left(
\begin{array}
[c]{ccc}%
a+j & d+f-b+j & d+j\\
d+c-b+j & d+f+2j &
\end{array}
;1\right)  =\nonumber
\end{gather}

\begin{gather}
=\frac{\Gamma(c)\Gamma(c+d-a-b)}{\Gamma(c-f)\Gamma(c+d+f-a-b)}\sum_{j\geq
0}K_{j}\left(  x|f,d\right)  \label{2F1-2b}\\
\times\frac{\left(  a\right)  ^{\left(  j\right)  }\left(  b\right)  ^{\left(
j\right)  }}{\left(  d+f+j-1\right)  ^{\left(  j\right)  }(c+d+f-a-b)^{\left(
j\right)  }}\label{2F1-2c}\\
\times~_{3}F_{2}\left(
\begin{array}
[c]{ccc}%
f+j & d+f-b+j & d+f-a+j\\
c+d+f-a-b+j & d+f+2j &
\end{array}
;1\right)  .\nonumber
\end{gather}

\begin{gather}
=\frac{\Gamma(d+f)\Gamma(c+d-a-b)}{\Gamma(d)\Gamma(c+d+f-a-b)}\sum_{j\geq
0}K_{j}\left(  x|f,d\right)  \frac{\left(  a\right)  ^{\left(  j\right)
}\left(  b\right)  ^{\left(  j\right)  }}{\left(  c\right)  ^{\left(
j\right)  }\left(  d+f+j-1\right)  ^{\left(  j\right)  }}\times\label{2F1-2d}%
\\
\frac{(d+f)^{\left(  2j\right)  }}{(d)^{\left(  j\right)  }%
(c+d+f-a-b)^{\left(  j\right)  }}~_{3}F_{2}\left(
\begin{array}
[c]{ccc}%
f+j & c-a & c-b\\
c+j & c+d+f-a-b+j &
\end{array}
;1\right)  .\nonumber
\end{gather}

\end{remark}

\begin{proof}
We get the following four equivalent expressions
\begin{align*}
&  _{3}F_{2}\left(
\begin{array}
[c]{ccc}%
a+j & b+j & f+j\\
c+j & d+f+2j &
\end{array}
;1\right) \\
&  =\frac{\Gamma(c+j)\Gamma(c+d-a-b)}{\Gamma(c-a)\Gamma(c+d-b+j)}~_{3}%
F_{2}\left(
\begin{array}
[c]{ccc}%
a+j & d+f-b+j & d+j\\
d+c-b+j & d+f+2j &
\end{array}
;1\right) \\
&  =\frac{\Gamma(c+j)\Gamma(c+d-a-b)}{\Gamma(c-f)\Gamma(c+d+f-a-b+j)}\times\\
&  ~_{3}F_{2}\left(
\begin{array}
[c]{ccc}%
f+j & d+f-b+j & d+f-a+j\\
c+d+f-a-b+j & d+f+2j &
\end{array}
;1\right) \\
&  =\frac{\Gamma(d+f+2j)\Gamma(c+d-a-b)}{\Gamma(d+f-a+j)\Gamma(d+c-b+j)}%
~_{3}F_{2}\left(
\begin{array}
[c]{ccc}%
a+j & c-b & c-f\\
c+j & d+c-b+j &
\end{array}
;1\right) \\
&  =\frac{\Gamma(d+f+2j)\Gamma(c+d-a-b)}{\Gamma(d+j)\Gamma(c+d+f-a-b+j)}\\
&  \times~_{3}F_{2}\left(
\begin{array}
[c]{ccc}%
f+j & c-a & c-b\\
c+j & c+d+f-a-b+j &
\end{array}
;1\right)
\end{align*}

Now we apply (\ref{Gn}) to get the desired identities. We proceed in a similar
way with (\ref{2F1-2}).
\end{proof}

\begin{lemma}
\label{part}a) By setting $f\allowbreak=\allowbreak c$ and assuming that
$c+d\allowbreak\neq\allowbreak1,$ we can obtain%
\begin{gather}
_{2}F_{1}\left(  a,b;c;x\right)  =\frac{\Gamma(c+d-1)\Gamma\left(
c+d-a-b\right)  }{\Gamma(c+d-a)\Gamma\left(  c+d-b\right)  }\label{2F1-2*}\\
\times\sum_{j\geq0}K_{j}\left(  x|c,d\right)  \frac{\left(  a\right)
^{\left(  j\right)  }\left(  b\right)  ^{\left(  j\right)  }\left(
c+d-1\right)  ^{\left(  j\right)  }(c+d+2j-1)}{\left(  c\right)  ^{\left(
j\right)  }\left(  d+c-a\right)  ^{\left(  j\right)  }\left(  d+c-b\right)
^{\left(  j\right)  }}.\nonumber
\end{gather}

b) Also, setting $d\allowbreak=f$ and assuming that $c\allowbreak
=\allowbreak(a+b+1)/2$ and $2f+1>a+b$ we have%
\begin{gather*}
_{2}F_{1}\left(  a,b;\frac{a+b+1}{2};x\right)  =\Gamma(\frac{1}{2}%
)\Gamma\left(  f+(1-a-b)/2\right)  \times\\
\sum_{j\geq0}K_{j}\left(  x|f,f\right)  \frac{\left(  a\right)  ^{\left(
j\right)  }\left(  b\right)  ^{\left(  j\right)  }\Gamma(f+1/2+j)}{\left(
\frac{(a+b+1)}{2}\right)  ^{\left(  j\right)  }\left(  2f+j-1\right)
^{\left(  j\right)  }\Gamma(\frac{(a+1+j)}{2})\Gamma(\frac{(b+1+j)}{2}%
)\Gamma(\frac{2f+1-a+j}{2})\Gamma(\frac{2f+1-b+j}{2})}.
\end{gather*}

c) Setting $f\allowbreak=\allowbreak2c-b-1$ and $d\allowbreak=\allowbreak
1+b+2a-2c$ and assuming that $f$ and $d$ are both positive we observe that
$f+d\allowbreak=\allowbreak2a$ and $b+f+1\allowbreak=\allowbreak2c.$ Hence,
after applying Watson's sum, we have%

\begin{gather*}
_{2}F_{1}\left(  a,b;c;x\right)  =\Gamma(\frac{1}{2})\Gamma(a+1-c)\times\\
\sum_{j\geq0}K_{j}(x|2c-b-1,1+b+2a-2c)\frac{\left(  b\right)  ^{\left(
j\right)  }(2a)^{\left(  2j\right)  }}{\left(  c\right)  ^{\left(  j\right)
}\left(  2a+j-1\right)  ^{\left(  j\right)  }(1+2a-c)^{\left(  j\right)  }}\\
\times\frac{\Gamma(a+\frac{1}{2}+j)\Gamma(c+j)}{\Gamma(\frac{b+1+j}{2}%
)\Gamma(c+\frac{j-1-b}{2})\Gamma(a+\frac{j+1-b}{2})\Gamma(a+1-c+\frac{j+b}%
{2})}.
\end{gather*}

d) Let us set $a\allowbreak=\allowbreak-n,$ $d\allowbreak=\allowbreak
b-c+1-n.$ We have then, after applying Pfaff--Saalsch\"{u}tz Balanced Sum
formula
\begin{gather}
\left(  c\right)  ^{\left(  n\right)  }~_{2}F_{1}\left(  -n,b;c;x\right)
=\left(  c\right)  ^{\left(  n\right)  }\sum_{j=0}^{n}(-1)^{j}\binom{n}%
{j}\frac{\left(  b\right)  ^{\left(  j\right)  }x^{j}}{\left(  c\right)
^{\left(  j\right)  }}\label{-n}\\
=(-1)^{n}\sum_{j=0}^{n}\frac{n!}{(n-j)!}K_{j}(x|f,d)\frac{\left(  b\right)
^{\left(  j\right)  }\left(  c-b\right)  ^{\left(  n-j\right)  }\left(
c-f\right)  ^{\left(  n-j\right)  }(d+f+2j-1)}{\left(  d+f+j-1\right)
^{\left(  n+1\right)  }}.\nonumber
\end{gather}

\end{lemma}

\begin{proof}
a) To get (\ref{2F1-2*}) we set $f\allowbreak=\allowbreak c$ in the either
(\ref{2F1-1}) or (\ref{2F1-1a}) then we apply Gauss summation formula and
(\ref{Gn}). Now we apply the following formula true for all complex $x$ not
equal to non-positive integers and $n\geq0$ and $x\neq1:$
\[
\left(  x\right)  ^{\left(  2n\right)  }/\left(  x+n-1\right)  ^{\left(
n\right)  }\allowbreak=\frac{\left(  x\right)  ^{\left(  n\right)  }\left(
x+2n-1\right)  }{x-1}.
\]%
\begin{gather*}
_{2}F_{1}\left(  a,b;c;x\right)  ==\frac{\Gamma(c+d)\Gamma\left(
d+c-a-b\right)  }{\Gamma(d+c-a)\Gamma\left(  d+c-b\right)  }\times\\
\sum_{j\geq0}K_{j}\left(  x|c,d\right)  \frac{\left(  a\right)  ^{\left(
j\right)  }\left(  b\right)  ^{\left(  j\right)  }\left(  c+d\right)
^{\left(  2j\right)  }}{\left(  c\right)  ^{\left(  j\right)  }\left(
d+c+j-1\right)  ^{\left(  j\right)  }\left(  d+c-a\right)  ^{\left(  j\right)
}\left(  d+c-b\right)  ^{\left(  j\right)  }}\\
=\frac{\Gamma(c+d-1)\Gamma\left(  c+d-a-b\right)  }{\Gamma(c+d-a)\Gamma\left(
c+d-b\right)  }\times\\
\sum_{j\geq0}K_{j}\left(  x|c,d\right)  \frac{\left(  a\right)  ^{\left(
j\right)  }\left(  b\right)  ^{\left(  j\right)  }\left(  c+d-1\right)
^{\left(  j\right)  }(c+d+2j-1)}{\left(  c\right)  ^{\left(  j\right)
}\left(  d+c-a\right)  ^{\left(  j\right)  }\left(  d+c-b\right)  ^{\left(
j\right)  }},
\end{gather*}

b) In this case we will apply the so-called Watson's sum given e.g. as formula
(16.4.7) of \cite{NIST} and get for $2f+1>a+b$%
\begin{align*}
_{2}F_{1}\left(  a,b;(a+b+1)/2;x\right)   &  =\sum_{j\geq0}K_{j}\left(
x|f,f\right)  \frac{\left(  a\right)  ^{\left(  j\right)  }\left(  b\right)
^{\left(  j\right)  }}{\left(  \frac{1}{2}(a+b+1)\right)  ^{\left(  j\right)
}\left(  2f+j-1\right)  ^{\left(  j\right)  }}\times\\
&  _{3}F_{2}\left(
\begin{array}
[c]{ccc}%
a+j & b+j & f+j\\
\frac{1}{2}(a+b+1)+j & 2f+2j &
\end{array}
;1\right)  .
\end{align*}
Now recall that using Watsons sum we get:%
\begin{gather*}
_{3}F_{2}\left(
\begin{array}
[c]{ccc}%
a+j & b+j & f+j\\
\frac{1}{2}(a+b+1)+j & 2f+2j &
\end{array}
;1\right)  =\\
\frac{2\sqrt{\pi}\Gamma(f+\frac{1}{2}+j)\Gamma\left(  \frac{1}{2}%
(a+b+1)+j\right)  \Gamma\left(  f+\frac{1}{2}-\frac{a+b}{2}\right)  }%
{\Gamma(\frac{a+1+j}{2})\Gamma(\frac{b+1+j}{2})\Gamma\left(  f+\frac{1}%
{2}+\frac{j}{2}-\frac{a}{2}\right)  \Gamma\left(  f+\frac{1}{2}+\frac{j}%
{2}-\frac{b}{2}\right)  }%
\end{gather*}

c) Setting $f\allowbreak=\allowbreak2c-b-1$ and $d\allowbreak=\allowbreak
1+b+2a-2c$ and assuming that $f$ and $d$ are both positive we observe that
$f+d\allowbreak=\allowbreak2a$ and $b+f+1\allowbreak=\allowbreak2c$ we get%
\begin{gather*}
_{3}F_{2}\left(
\begin{array}
[c]{ccc}%
a+j & b+j & 2c-b-1+j\\
c+j & 2a+2j &
\end{array}
;1\right)  =\\
\frac{\sqrt{\pi}\Gamma(a+\frac{1}{2}+j)\Gamma(c+j)\Gamma(a+1-c)}{\Gamma
(\frac{b+1+j}{2})\Gamma(c+\frac{j-1-b}{2})\Gamma(a+\frac{j+1-b}{2}%
)\Gamma(a+1-c+\frac{j+b}{2})}=
\end{gather*}

d) By setting $a\allowbreak=\allowbreak-n$ and $d\allowbreak=\allowbreak
b-c+1-n$ we make
\[
_{3}F_{2}\left(
\begin{array}
[c]{ccc}%
a+j & b+j & f+j\\
c+j & d+f+2j &
\end{array}
;1\right)
\]
a balanced sum and a Pfaff--Saalsch\"{u}tz Balanced Sum formula (\cite{NIST}%
,16.4.3) can be applied yielding%
\begin{align*}
&  _{3}F_{2}\left(
\begin{array}
[c]{ccc}%
-n+j & b+j & f+j\\
c+j & \allowbreak b-c+1-n+f+2j &
\end{array}
;1\right)  \\
&  =\frac{\left(  c-b\right)  ^{n-j}\left(  c-f\right)  ^{\left(  n-j\right)
}}{\left(  c+j\right)  ^{\left(  n-j\right)  }\left(  c-b-f-j\right)
^{\left(  n-j\right)  }}.
\end{align*}
Finally, we notice that $\left(  c\right)  ^{\left(  j\right)  }\left(
c+j\right)  ^{\left(  n-j\right)  }\allowbreak=\allowbreak\left(  c\right)
^{\left(  n\right)  }.$ Further, we notice that $\left(  c-b-f-j\right)
\allowbreak=\allowbreak\left(  -d-f-n-j+1\right)  $ and consequently that
$\left(  -d-n+1-f-j\right)  ^{\left(  n-j\right)  }\allowbreak$\newline%
$=\allowbreak\left(  -d-f-2j+1-(n-j)\right)  ^{\left(  n-j\right)
}\allowbreak=\allowbreak\left(  -1\right)  ^{n-j}\left(  d+f+2j\right)
^{\left(  n-j\right)  },$ by (\ref{-x}). Now it is easy to notice that t and
$\left(  d+f+j-1\right)  ^{\left(  j\right)  }\left(  d+f+2j\right)  ^{\left(
n-j\right)  }\allowbreak=\allowbreak\left(  d+f+j-1\right)  ^{\left(
n+1\right)  }\allowbreak/\allowbreak\left(  d+f+2j-1\right)  $.
\end{proof}

\begin{corollary}
[Special values]i) Let us put $x\allowbreak=\allowbreak0$ and $x\allowbreak
=\allowbreak1$ in (\ref{GHK1}) and let us use \ref{K01}. We get then the
following general identities%
\begin{gather}
1=\sum_{j\geq0}\frac{(-1)^{j}}{\left(  f\right)  ^{\left(  j\right)  }}%
\frac{\prod_{s=1}^{n}\left(  a_{s}\right)  ^{\left(  j\right)  }}{\prod
_{s=1}^{m}\left(  b_{s}\right)  ^{\left(  j\right)  }\left(  d+f+j-1\right)
^{\left(  j\right)  }}\nonumber\\
\times~_{n+1}F_{m+1}\left(
\begin{array}
[c]{cccc}%
a_{1}+j & \ldots & a_{n}+j & f+j\\
b_{1}+j & \ldots & b_{m}+j & \left(  d+f+2j\right)
\end{array}
;1\right)  .\nonumber
\end{gather}
and
\begin{gather*}
_{n}F_{m}\left(
\begin{array}
[c]{ccc}%
a_{1} & \ldots & a_{n}\\
b_{1} & \ldots & b_{m}%
\end{array}
;1\right)  =\\
\sum_{j\geq0}\frac{1}{\left(  d\right)  ^{\left(  j\right)  }}\frac
{\prod_{s=1}^{n}\left(  a_{s}\right)  ^{\left(  j\right)  }}{\prod_{s=1}%
^{m}\left(  b_{s}\right)  ^{\left(  j\right)  }\left(  d+f+j-1\right)
^{\left(  j\right)  }}\\
\times~_{n+1}F_{m+1}\left(
\begin{array}
[c]{cccc}%
a_{1}+j & \ldots & a_{n}+j & f+j\\
b_{1}+j & \ldots & b_{m}+j & \left(  d+f+2j\right)
\end{array}
;1\right)  .
\end{gather*}

ii) For $n\allowbreak=\allowbreak2$ and $m\allowbreak=\allowbreak1$ and for
$\operatorname{Re}(c-a-b)>0,$ $d>0$ the following forms are obtained by
putting $x\allowbreak=\allowbreak0$ and $x\allowbreak=\allowbreak1$ in
(\ref{2F1-2*}).
\begin{align}
1 &  =\frac{\Gamma(c+d-1)\Gamma\left(  c+d-a-b\right)  }{\Gamma(c+d-a)\Gamma
\left(  c+d-b\right)  }\label{SC1}\\
&  \times\sum_{j\geq0}\left(  -1\right)  ^{j}\frac{\left(  a\right)  ^{\left(
j\right)  }\left(  b\right)  ^{\left(  j\right)  }\left(  c+d-1\right)
^{\left(  j\right)  }(c+d+2j-1)}{j!\left(  d+c-a\right)  ^{\left(  j\right)
}\left(  d+c-b\right)  ^{\left(  j\right)  }},\nonumber\\
\frac{\Gamma(c)\Gamma(c-a-b)}{\Gamma(c-a)\Gamma(c-b)} &  =\frac{\Gamma
(c+d-1)\Gamma\left(  c+d-a-b\right)  }{\Gamma(c+d-a)\Gamma\left(
c+d-b\right)  }\label{SC2}\\
&  \times\sum_{j\geq0}\frac{\left(  d\right)  ^{\left(  j\right)  }\left(
a\right)  ^{\left(  j\right)  }\left(  b\right)  ^{\left(  j\right)  }\left(
c+d-1\right)  ^{\left(  j\right)  }(c+d+2j-1)}{\left(  j!\right)  \left(
c\right)  ^{\left(  j\right)  }\left(  d+c-a\right)  ^{\left(  j\right)
}\left(  d+c-b\right)  ^{\left(  j\right)  }}.\nonumber
\end{align}

iii) For and we get the following identities, true for all complex $c,b,f$
such that $c-b-f$ is not equal to a non-negative integerer less or equal to
$n+1.$%
\begin{align}
\left(  c\right)  ^{n}\allowbreak &  =\allowbreak\sum_{j=0}^{n}(-1)^{j}%
\binom{n}{j}\frac{\left(  b\right)  ^{\left(  j\right)  }\left(  f\right)
^{\left(  j\right)  }\left(  c-b\right)  ^{\left(  n-j\right)  }\left(
c-f\right)  ^{\left(  n-j\right)  }(c-b-f+n-2j)}{\left(  c-b-f-j\right)
^{\left(  n+1\right)  }},\label{PSS1}\\
1 &  =\sum_{j=0}^{n}\left(  -1\right)  ^{j}\binom{n}{j}\frac{\left(  b\right)
^{\left(  j\right)  }\left(  c-f\right)  ^{\left(  n-j\right)  }\left(
c-b-f+n-2j\right)  }{\left(  c-b-f-j\right)  ^{\left(  n+1\right)  }%
}.\label{PSS2}%
\end{align}

\end{corollary}

\begin{proof}
Only iii) requires the proof. To get these identities, we consider (\ref{-n})
with $a\allowbreak=\allowbreak-n,$ and $d\allowbreak=\allowbreak\allowbreak
b-c+1-n$ and set respectively firstly $x\allowbreak=\allowbreak0$ and then
$x\allowbreak=\allowbreak1.$ Then recall that $\left(  c\right)  ^{\left(
n\right)  }~_{2}F_{1}\left(  -n,b;c;0\right)  \allowbreak=\allowbreak\left(
c\right)  ^{\left(  n\right)  }$ while $\left(  c\right)  ^{\left(  n\right)
}~_{2}F_{1}\left(  -n,b;c;1\right)  \allowbreak=\allowbreak\left(  c\right)
^{\left(  n\right)  }\frac{\left(  c-b\right)  ^{\left(  n\right)  }}{\left(
c\right)  ^{\left(  n\right)  }}\allowbreak=\allowbreak\left(  c-b\right)
^{\left(  n\right)  }.$ Then we recall (\ref{K01}) and immediately get
\ref{PSS1}. To get (\ref{PSS2}) we have to recall that $\left(
b+1-c-n\right)  ^{\left(  j\right)  }\allowbreak=\allowbreak(-1)^{j}\left(
c-b+n-j\right)  ^{\left(  j\right)  }$ by (\ref{-x}). Next we notice that
$\left(  c-b+n-j\right)  ^{\left(  j\right)  }\left(  c-b\right)  ^{\left(
n-j\right)  }\allowbreak=\allowbreak\left(  c-b\right)  ^{\left(  n\right)  }$
and we cancel out this factor on both sides.
\end{proof}

\begin{remark}
To make the reader realize how nontrivial the obtained identities are, let us
consider identities (\ref{SC1}) and (\ref{SC2}) for $a\allowbreak
=\allowbreak-n.$ Namely, after some elementary algebra (\ref{SC1}) becomes
after denoting $c+d\allowbreak=\allowbreak\alpha$ and $b\allowbreak
=\allowbreak\beta,$ for all $n,$ $\alpha\neq0,$ $\beta\neq\beta$ we get%
\[
\frac{1}{\left(  \alpha-\beta\right)  ^{\left(  n\right)  }}=\sum_{j\geq
0}\binom{n}{j}\left(  1+\frac{j}{\alpha+j-1}\right)  \frac{\left(
\beta\right)  ^{\left(  j\right)  }\left(  \alpha\right)  ^{\left(  j\right)
}}{\left(  a\right)  ^{\left(  n+j\right)  }\left(  a-\beta\right)  ^{\left(
j\right)  }}.
\]
Identity (\ref{SC2}) for $a\allowbreak=\allowbreak-n$ and some algebra becomes
after denoting $c-b\allowbreak=\allowbreak\alpha,$ $c\allowbreak
=\allowbreak\gamma,$ $c+d-b\allowbreak=\allowbreak\beta$%
\begin{align*}
\frac{\left(  \alpha\right)  ^{\left(  n\right)  }}{\left(  \beta\right)
^{\left(  n\right)  }\left(  \gamma\right)  ^{\left(  n\right)  }} &
=\sum_{j=0}^{n}(-1)^{j}\binom{n}{j}\left(  1+\frac{j}{\beta+\gamma
-\alpha++j-1}\right)  \\
&  \times\frac{\left(  \beta-\alpha\right)  ^{\left(  j\right)  }\left(
\gamma-\alpha\right)  ^{\left(  j\right)  }(\beta+\gamma-\alpha)^{\left(
j\right)  }}{\left(  \gamma\right)  ^{\left(  j\right)  }\left(  \beta\right)
^{\left(  j\right)  }\left(  \beta+\gamma-\alpha\right)  ^{\left(  n+j\right)
}},
\end{align*}
which is true for $n\geq0,$ $\gamma\neq0,$ $\beta\neq0,$ $\alpha\neq
\beta+\gamma.$
\end{remark}

\begin{corollary}
[Integrals]1. For all $f,d>0$ we have
\begin{align}
&  \int_{0}^{1}~_{2}F_{1}\left(  a,b;c;x\right)  K_{j}\left(  x|f,d\right)
f(x|f,d)dx\label{IFK1}\\
&  =\frac{\left(  f\right)  ^{\left(  j\right)  }\left(  d\right)  ^{\left(
j\right)  }\left(  a\right)  ^{\left(  j\right)  }\left(  b\right)  ^{\left(
j\right)  }}{j!\left(  f+d\right)  ^{\left(  2j\right)  }\left(  c\right)
^{\left(  j\right)  }}~_{3}F_{2}\left(  a+j,b+j,f+j;c+j,d+f+2j;1\right)
.\nonumber
\end{align}

2. In particular if $c=f$ is a positive we get%
\begin{align*}
&  \int_{0}^{1}~_{2}F_{1}\left(  a,b;c;x\right)  K_{j}\left(  x|c,d\right)
f(x|c,d)dx\\
&  =\frac{\left(  d\right)  ^{\left(  j\right)  }\left(  a\right)  ^{\left(
j\right)  }\left(  b\right)  ^{\left(  j\right)  }}{j!\left(  d+c-a\right)
^{\left(  j\right)  }\left(  d+c-b\right)  ^{\left(  j\right)  }}\frac
{\Gamma\left(  c+d\right)  \Gamma\left(  d+c-a-b\right)  }{\Gamma
(d+c-a)\Gamma(d+c-b)}.
\end{align*}

3. For all complex $a,b,c,d$ such that $\operatorname{Re}(c-a-b)>0,$ we have%
\begin{gather*}
\frac{\Gamma(c)\Gamma(c-a-b)\Gamma(c+d-a)\Gamma\left(  c+d-b\right)  }%
{\Gamma(c-a)\Gamma(c-b)\Gamma(c+d)\Gamma\left(  c+d-a-b\right)  }\\
=~_{4}F_{3}\left(
\begin{array}
[c]{cccc}%
a & b & d & c+d-1\\
c & c+d-a & c+d-b &
\end{array}
;1\right) \\
+\frac{2abd(c+d-1)}{c(c+d-a)(c+d-b)}~_{4}F_{3}\left(
\begin{array}
[c]{cccc}%
a+1 & b+1 & d+1 & c+d\\
c+1 & c+d-a+1 & c+d-b+1 &
\end{array}
;1\right)
\end{gather*}

and%
\begin{gather*}
\frac{\Gamma(c+d-a)\Gamma\left(  c+d-b\right)  }{\Gamma(c+d)\Gamma\left(
c+d-a-b\right)  }=~_{3}F_{2}\left(
\begin{array}
[c]{ccc}%
a & b & c+d-1\\
c+d-a & c+d-b &
\end{array}
;-1\right) \\
-\frac{2ab}{(c+d-a)(c+d-b)}~_{3}F_{2}\left(
\begin{array}
[c]{ccc}%
a+1 & b+1 & c+d\\
c+d+1 & c+d+1 &
\end{array}
;-1\right)
\end{gather*}

4) For all non-negative $n,$ $m$ and complex $a,b,c,d$ such that
$\operatorname{Re}(c-a-b)>0,$ we have%
\begin{align*}
&  _{3}F_{2}\left(
\begin{array}
[c]{ccc}%
a+n & b+n & 1+n\\
c+n & 1 & 1+m+2n
\end{array}
;1\right) \\
&  =\frac{(2n+m)!\left(  c\right)  ^{\left(  n\right)  }}{n!\left(  a\right)
^{\left(  n\right)  }\left(  b\right)  ^{\left(  n\right)  }}\sum_{j=0}%
^{n}\left(  -1\right)  ^{j}\binom{n}{j}\frac{\left(  n+m\right)  ^{\left(
j\right)  }\left(  c-j-m\right)  ^{\left(  j+m\right)  }}{\left(
a-j-m\right)  ^{\left(  j+m\right)  }\left(  b-j-m\right)  ^{\left(
j+m\right)  }}\times\\
&  \left(  \frac{\Gamma(c-j-m)\Gamma(c-a-b+m+j)}{\Gamma(c-a)\Gamma(c-b)}%
-\sum_{k=0}^{j+m-1}\frac{\left(  a-m-j\right)  ^{\left(  k\right)  }\left(
b-m-j\right)  ^{\left(  k\right)  }}{k!\left(  c-m-j\right)  ^{k}}\right)
\end{align*}

\end{corollary}

\begin{proof}
1) We use (\ref{2F1-2}) and (\ref{SQR}), and then get%
\begin{gather*}
\int_{0}^{1}F_{1}\left(  a,b;c;x\right)  K_{j}\left(  x|f,d\right)
f(x|f,d)dx=\frac{\left(  f\right)  ^{\left(  j\right)  }\left(  d\right)
^{\left(  j\right)  }}{j!\left(  f+d+2j-1\right)  \left(  f+d\right)
^{\left(  j-1\right)  }}\\
\times\frac{\left(  a\right)  ^{\left(  j\right)  }\left(  b\right)  ^{\left(
j\right)  }}{\left(  c\right)  ^{\left(  j\right)  }\left(  d+f+j-1\right)
^{\left(  j\right)  }}~_{3}F_{2}\left(  a+j,b+j,f+j;c+j,d+f+2j;1\right)  .
\end{gather*}
Now we observe that $\left(  f+d+2j-1\right)  \left(  f+d\right)  ^{\left(
j-1\right)  }\left(  d+f+j-1\right)  ^{\left(  j\right)  }\allowbreak
=\allowbreak\left(  f+d\right)  ^{\left(  2j\right)  }.$

2. Setting $c\allowbreak=\allowbreak f$ and using 1., we get%
\begin{align*}
&  \int_{0}^{1}F_{1}\left(  a,b;c;x\right)  K_{j}\left(  x|c,d\right)
f(x|c,d)dx\\
&  =\frac{\left(  d\right)  ^{\left(  j\right)  }\left(  a\right)  ^{\left(
j\right)  }\left(  b\right)  ^{\left(  j\right)  }}{j!\left(  c+d\right)
^{\left(  2j-1\right)  }}~_{2}F_{1}\left(  a+j,b+j;d+c+2j;1\right)  .
\end{align*}
To proceed further we use Gauss summation theorem and get
\begin{align*}
&  \int_{0}^{1}F_{1}\left(  a,b;c;x\right)  K_{j}\left(  x|c,d\right)
f(x|c,d)dx\\
&  =\frac{\left(  d\right)  ^{\left(  j\right)  }\left(  a\right)  ^{\left(
j\right)  }\left(  b\right)  ^{\left(  j\right)  }}{j!\left(  c+d\right)
^{\left(  2j\right)  }}\frac{\Gamma\left(  d+c+2j\right)  \Gamma\left(
d+c-a-b\right)  }{\Gamma(d+c-a+j)\Gamma(d+c-b+j)}.
\end{align*}
Now we apply (\ref{Gn}). In particular, we observe that $\Gamma\left(
d+c+2j\right)  \allowbreak=$\newline$\allowbreak\left(  d+c\right)  ^{\left(
2j\right)  }\Gamma\left(  d+c\right)  $ and we cancel out $\left(  d+c\right)
^{\left(  2j\right)  }$.

3. We use the fact that under our restrictions on parameters we have
(\ref{K01}) and
\[
_{2}F_{1}\left(  a,b;c;0\right)  =1,~_{2}F_{1}\left(  a,b;c;1\right)
=\frac{\Gamma(c)\Gamma(c-a-b)}{\Gamma(c-a)\Gamma(c-b)}.
\]
Now, we insert these values in (\ref{2F1-2*}) and proceed as follows%
\begin{gather*}
1=\frac{\Gamma(c+d-1)\Gamma\left(  c+d-a-b\right)  }{\Gamma(c+d-a)\Gamma
\left(  c+d-b\right)  }\\
\times\sum_{j\geq0}\frac{(-1)^{j}\left(  c\right)  ^{\left(  j\right)  }}%
{j!}\frac{\left(  a\right)  ^{\left(  j\right)  }\left(  b\right)  ^{\left(
j\right)  }\left(  c+d-1\right)  ^{\left(  j\right)  }(c+d+2j-1)}{\left(
c\right)  ^{\left(  j\right)  }\left(  d+c-a\right)  ^{\left(  j\right)
}\left(  d+c-b\right)  ^{\left(  j\right)  }}.
\end{gather*}
Now, it is enough to notice that the sum can be split into two sums. The first
one is equal to
\begin{align*}
&  \left(  c+d-1\right)  \sum_{j\geq0}\frac{(-1)^{j}}{j!}\frac{\left(
a\right)  ^{\left(  j\right)  }\left(  b\right)  ^{\left(  j\right)  }\left(
c+d-1\right)  ^{\left(  j\right)  }}{\left(  d+c-a\right)  ^{\left(  j\right)
}\left(  d+c-b\right)  ^{\left(  j\right)  }}\\
&  =(c+d-1)_{3}F_{2}\left(
\begin{array}
[c]{ccc}%
a & b & c+d-1\\
c+d-a & c+d-b &
\end{array}
;-1\right)  .
\end{align*}
The second is equal to%
\begin{align*}
&  2\sum_{j\geq1}\frac{(-1)^{j}}{\left(  j-1\right)  !}\frac{\left(  a\right)
^{\left(  j\right)  }\left(  b\right)  ^{\left(  j\right)  }\left(
c+d-1\right)  ^{\left(  j\right)  }}{\left(  d+c-a\right)  ^{\left(  j\right)
}\left(  d+c-b\right)  ^{\left(  j\right)  }}\\
&  =-2\frac{ab(c+d-1)}{(c+d-a)(c+d-b)}~_{3}F_{2}\left(
\begin{array}
[c]{ccc}%
a+1 & b+1 & c+d\\
c+d-a+1 & c+d-b+1 &
\end{array}
;-1\right)  .
\end{align*}
In the last line we used the fact that $\left(  x\right)  ^{\left(
n+1\right)  }\allowbreak=\allowbreak x\left(  x+1\right)  ^{\left(  n\right)
}$.

In the case of $x\allowbreak=\allowbreak1$ we proceed in the similar way.

4) We find $\int_{0}^{1}~_{2}F_{1}\left(  a,b;c;x\right)  K_{j}\left(
x|f,d\right)  f(x|f,d)dx$ in two ways. The first one follows formulae
(\ref{2F1-1}) and (\ref{SQR}) getting (\ref{IFK1}). On the other hand we
notice that we also have%
\[
\int_{0}^{1}K_{n}(x|f,d)f\left(  x|f,d\right)  dx=\frac{\left(  f\right)
^{\left(  k\right)  }\left(  d\right)  ^{\left(  n\right)  }}{n!(f+d)^{\left(
k\right)  }}\sum_{m=0}^{n}(-1)^{m}\binom{n}{m}\frac{\left(  f+d+n-1\right)
^{\left(  m\right)  }}{\left(  f+d+k\right)  ^{\left(  m\right)  }},
\]
and consequently%
\begin{gather*}
\int_{0}^{1}~_{2}F_{1}\left(  a,b;c;x\right)  K_{j}\left(  x|f,d\right)
f(x|f,d)dx=\\
\left(  d\right)  ^{\left(  n\right)  }\sum_{k\geq0}\frac{\left(  a\right)
^{\left(  k\right)  }\left(  b\right)  ^{\left(  k\right)  }\left(  f\right)
^{\left(  k\right)  }}{\left(  c\right)  ^{\left(  k\right)  }k!}\sum
_{m=0}^{n}(-1)^{m}\frac{\left(  f+d+n-1\right)  ^{\left(  m\right)  }%
}{m!(n-m)!\left(  f+d\right)  ^{\left(  k+m\right)  }}..
\end{gather*}
After changing the order of summation setting $f\allowbreak=\allowbreak1$ and
$d\allowbreak=\allowbreak m$ and further necessary simplifications we get%
\begin{gather*}
\int_{0}^{1}~_{2}F_{1}\left(  a,b;c;x\right)  K_{j}\left(  x|1,m\right)
f(x|1,m)dx=\\
m\sum_{j=0}^{n}\left(  -1\right)  ^{j}\frac{(m+n+j-1)!\left(  c-j-m\right)
^{\left(  j+m\right)  }}{j!(n-j)!\left(  a-j-m\right)  ^{\left(  j+m\right)
}\left(  b-j-m\right)  ^{\left(  j+m\right)  }}\\
\times\sum_{k\geq0}\frac{\left(  a-j-m\right)  ^{\left(  k+j+m\right)
}\left(  b-j-m\right)  ^{\left(  k+j+m\right)  }}{\left(  c-j-m\right)
^{\left(  k+j+m\right)  }\left(  m+j+k\right)  !}.
\end{gather*}
Now it is enough to notice that the second sum is equal to
\[
_{2}F_{1}\left(  a-j-m,b-j-m;c-j-m;1\right)  \allowbreak-\allowbreak\sum
_{k=0}^{j+m-1}\frac{\left(  a-j-m\right)  ^{\left(  k\right)  }\left(
b-j-m\right)  ^{\left(  k\right)  }}{k!\left(  c-j-m\right)  ^{\left(
k\right)  }}.
\]
Now it remains to apply Gauss summation formula.
\end{proof}

\begin{remark}
Recall the formula (\ref{FK2}), and let us set on the left-hand-side of
(\ref{2F1-2}) $a\allowbreak->\allowbreak-n,$ $b\allowbreak->$ \allowbreak
$a+b+n-1,$ $c\allowbreak->\allowbreak a,$ we have
\begin{gather*}
\left(  a\right)  ^{\left(  n\right)  }(-1)^{n}K_{n}(x|a,b)/n!=\sum_{j\geq
0}K_{j}\left(  x|f,d\right)  \frac{\left(  -n\right)  ^{\left(  j\right)
}\left(  a+b+n-1\right)  ^{\left(  j\right)  }}{\left(  a\right)  ^{\left(
j\right)  }\left(  d+f+j-1\right)  ^{\left(  j\right)  }}\\
\times~_{3}F_{2}\left(
\begin{array}
[c]{ccc}%
-n+j & a+b+n-1+j & f+j\\
a+j & d+f+2j &
\end{array}
;1\right)  .
\end{gather*}
Now, since we have%
\begin{align*}
&  _{3}F_{2}\left(
\begin{array}
[c]{ccc}%
-n+j & a+b+n-1+j & f+j\\
a+j & d+f+2j &
\end{array}
;1\right) \\
&  =\sum_{k=0}^{n-j}(-1)^{k}\binom{n-j}{k}\frac{\left(  a+b+n-1+j\right)
^{\left(  k\right)  }\left(  f+j\right)  ^{\left(  k\right)  }}{\left(
a+j\right)  ^{\left(  k\right)  }\left(  d+f+2j\right)  ^{\left(  k\right)  }%
},
\end{align*}
we can see that we get, after necessary simplification, the set of connection
coefficients between $K_{n}(x|a,b)$ and $K_{j}(x|f,d).$ These coefficients
were already obtained, analyzed and simplified recently in \cite{Szab25}, but,
what is more interesting, were obtained a long time ago by other methods by R.
Askey in \cite{Ask75}(7.23, 7.32, 7.33). The Jacobi polynomials we are using
have a slightly different definition, so it's important to remember that.
\end{remark}

\end{document}